\newfont{\Bbb}{msbm10 scaled\magstephalf}
\documentclass[12pt,reqno]{amsart}
\usepackage{amsmath, amsthm, amscd, amsfonts}
\usepackage{graphicx}
\usepackage{amssymb}
\usepackage{mathrsfs}  
 \newtheorem{thm}{Theorem}[section]
 \newtheorem{cor}[thm]{Corollary}
 \newtheorem{lem}[thm]{Lemma}
 \newtheorem{prop}[thm]{Proposition}
 
 \theoremstyle{definition}
 
 \theoremstyle{remark}

 \numberwithin{equation}{section}
\def\d{\mathrm{d}}
\def\vp{\propto}
\begin{document}
\title[Compact intertwining relations]
 {Compact intertwining relations for composition operators on $H^\infty$ and the Bloch spaces}

\author[C. Z. Tong, Z. H. Zhou and C. Yuan]{Ce-Zhong Tong \and Ze-Hua Zhou$^*$ \and Cheng Yuan}

\address{\newline Ce-Zhong Tong\newline Department of Mathematics,
Hebei University Technology, Tianjin 300401, P.R. China.}
\email{cezhongtong@hotmail.com}

\address{\newline Ze-Hua Zhou\newline Department of Mathematics,
Tianjin University, Tianjin 300350, P.R. China.}
\email{zehuazhoumath@aliyun.com;zhzhou@tju.edu.cn}

\address{\newline Cheng Yuan\newline Institute of Mathematics, School of Science,
Tianjin University of Technology and
Education, Tianjin 300222, P.R. China}
\email{yuancheng1984@163.com}



\date{}
\thanks{\noindent $^*$ Corresponding author.\\
The work was supported by the National Natural Science Foundation of China
(Grant Nos. 11301132, 11301373, 11371276, 11201331, 11171087) and
Natural Science Foundation of Hebei Province (Grant No. A2013202265).}

\begin{abstract}
On the space of bounded analytic functions and the Bloch space on the unit disk, we study the compact intertwining relations
for composition operators, whose intertwining operators are Volterra type operators.
Further, we consider the compact intertwining relations, which are between the
whole collection of composition operators and some Volterra operator, and the whole
collection of bounded Volterra operators and some composition operator.

\end{abstract}

\maketitle
\section{Introduction}

If $X$ and $Y$ are two Banach spaces, the symbol $\mathscr{B}(X, Y)$ denotes the collection of all
bounded linear operators from $X$ to $Y$.
Let $\mathcal{K}(X, Y)$ be the collection of
all compact elements of $\mathscr{B}(X,Y)$, and let $\mathscr{Q}(X, Y)$ be the quotient set
$\mathscr{B}(X, Y)/\mathcal{K}(X, Y)$. 

For linear operators $A\in\mathscr{B}(X,X)$, $B\in\mathscr{B}(Y,Y)$ and $T\in\mathscr{B}(X, Y)$,
the phrase ``$T$ intertwines
$A$ and $B$
in $\mathscr{Q}(X, Y)$" (or ``$T$ intertwines $A$ and
$B$ compactly") means that
\begin{equation}
  \label{a4}TA=BT\mod{\mathcal{K}(X, Y)}\quad\mbox{with}\quad T\neq0.
\end{equation}
The notation $A\propto_K B$ $(T)$ represents the relation in equation \eqref{a4}. In fact, if $T$
is an invertible operator on $X$, then the relation $\propto_K$ is
symmetric.

Recall that the essential norm of a bounded linear operator $T$ is the distance
from $T$ to the compact operators, that is, $$\|T\|_e=\inf\{\|T-K\|: K\mbox{ is compact}\}.$$
Notice that $\|T\|_e=0$ if and only if $T$ is compact. So estimates on $\|T\|_e$ lead
to conditions for $T$ to be compact.

Let $\mathbb{D}$ be the unit disk in the complex plane. Denote by
$H(\mathbb{D})$ the class of all holomorphic functions on
$\mathbb{D}$, and $S(\mathbb{D})$ the collection of all the
holomorphic self-mappings of $\mathbb{D}$. Every $\varphi\in
S(\mathbb{D}) $ induces a composition operator $C_\varphi$ defined
by $C_\varphi f=f\circ\varphi$ for $f\in H(\mathbb{D})$.

Let $g\in H(\mathbb{D})$. The Volterra operator $J_g$ is defined by
$$J_gf(z)=\int_0^z f(\zeta)g'(\zeta) d \zeta,\,z\in\mathbb{D},\;\; f\in H(\mathbb{D});$$ and
another integral operator $I_g$ is defined by
$$I_gf(z)=\int_0^z f'(\zeta)g(\zeta) d \zeta,\,z\in\mathbb{D},\;\; f\in H(\mathbb{D}).$$

The notation $H^\infty(\mathbb{D})$ represents the algebra of bounded
holomorphic functions with $\|\cdot\|_\infty$ as its supreme
norm. For $\alpha>0,\beta\in\mathbb{R}$, the Bloch type space $\mathcal{B}_{\alpha,\log^\beta}$ consists of all $f\in
H(\mathbb{D})$ such that
$$\|f\|_{*}:=\sup\limits_{z\in\mathbb{D}}(1-|z|^2)^\alpha\log^\beta\frac{2}{1-|z|^2}|f'(z)|<\infty.$$
Then $\|\cdot\|_*$ is a complete semi-norm on $\mathcal{B}_{\alpha,\log^\beta}$, which is M\"{o}bius
invariant.
We denote the Banach space associated to $\mathcal B_{\alpha,\log^\beta}$ by $\tilde{\mathcal B}_{\alpha,\log^\beta}$,
where the norm is given by the formula
$$\|f\|_{\tilde{\mathcal{B}}_{\alpha,\log^\beta}}=|f(0)|+\|f\|_*.$$
The little Bloch space, denoted by $\mathcal{B}_{\alpha,\log^\beta;0}$, consists of
$f\in\mathcal{B}_{\alpha,\log^\beta}$ for which
$$\lim\limits_{|z|\rightarrow1}(1-|z|^2)^\alpha\log^\beta\frac{2}{1-|z|^2}|f'(z)|=0.$$
In this paper, we abbreviate $\mathcal{B}=\mathcal{B}_{1,\log^0}$ and $\mathcal{B}_{\log}=\mathcal{B}_{1,\log^1}$.

Composition operators were studied intensively in the past a few decades.
A lot of efforts have been made on characterizing bounded and compact composition operators on spaces of analytic functions,
 for example, \cite{S1} for Hardy space and \cite{MM} for Bloch spaces.
Interested readers may refer to books \cite{CM,Sha,Z2} and some recent papers \cite{ZC,ZS,ZZ} to learn
much more on this subject.

The discussion of $J_g$ first
arose in connection with semigroups of composition operators, and readers may refer to \cite{SZ} for background.
Recently, the problem of characterizing the boundedness and compactness of $J_g$ and $I_g$ on various spaces
of analytic functions has attracted considerable attention.
For example, the boundedness of $J_g$ on Hardy spaces, Bergman spaces, BMOA space, Bloch space
and $\mathcal{Q}_p$ space are characterized in \cite{AC,AS,SZ,X2,X3,X1}, respectively. The
same problems for the product of composition and Volterra operators
on kinds of function spaces on the open unit disk of the plane have
also been discussed, see some recent papers \cite{LS,LS6,LS7'}.

Based on these results, we consider the composition operator
$C_\varphi: X\rightarrow X$, and the integral-type operator $V_g\;
(=J_g\;\mbox{or}\;I_g): X\rightarrow\mathcal{B}$ where $X$ represents
$H^\infty$ or $\mathcal{B}$. 
We are interested in the compact intertwining relations
\begin{equation}
  \label{13} C_\varphi\mid_{X}\vp_K C_\varphi\mid_{\mathcal{B}}\;\;(V_g\mid_{X\rightarrow\mathcal{B}})
\end{equation}

If \eqref{13} holds for any
$\varphi\in S(\mathbb{D})$ and $g\in H(\mathbb{D})$,
we may also say that $C_\varphi$ and $V_g$
\emph{essentially commute}.
It is of special interest to us to decide when a given Volterra operator essentially commutes with every composition
operator and when a given composition operator essentially commutes with every Volterra operator.

Two main Questions in this paper are:

\begin{description}
  \item[(Q1)] what properties should a non-constant $g$ have, if
$$C_\varphi\mid_{X}\vp_K C_\varphi\mid_{\mathcal{B}}\quad(V_g\mid_{X\rightarrow\mathcal{B}})$$ holds for every $\varphi\in S(\mathbb{D})$; and
  \item[(Q2)] what properties should a $\varphi\in S(\mathbb{D})$ have, if
$$C_\varphi\mid_{X}\vp_K C_\varphi\mid_{\mathcal{B}}\quad(V_g\mid_{X\rightarrow\mathcal{B}})$$ holds for every bounded
$V_g$?
\end{description}

For simplicity, if $g$ satisfies conditions in \textbf{(Q1)}, we write $C\mid_{X}\vp_K C\mid_{\mathcal{B}}$ $(V_g)$;
if $\varphi$ satisfies conditions in \textbf{(Q2)}, we write $C_\varphi\mid_{X}\vp_K C_\varphi\mid_{\mathcal{B}}$
$(V\mid_{X\rightarrow\mathcal{B}})$.
By the way, the collections of $g$ satisfying conditions similar as \textbf{(Q1)} was called the
\emph{universal set} of $V_g$ by the authors in \cite{TZ2,TZ}.
Our use of the term ``universal set" should not be confused with the notion of ``universal set" which appears in the
dynamical theory of linear operators.

In the following discussion, we write $A\lesssim B$ if there exists an absolute constant $C$ such that $A\leq C\cdot B$,
and $A\approx B$ represents $A\lesssim B$ and $B\lesssim A$.

%
%
%

\section{Preliminaries}

Before the discussion of our main results, we need some preliminary
notation and propositions. We state them without proof.

For $\varphi\in S(\mathbb{D})$, denote
the \emph{Schwarz derivative} of $\varphi$ by
$$\varphi^\#(z):=\frac{1-|z|^2}{1-|\varphi(z)|^2}\varphi'(z).$$ From the
Schwarz Lemma we know that $|\varphi^\#(z)|\leq1$, and the equality
holds if and only if $\varphi$ is an automorphism of the unit disk.
The following lemma characterizes bounded and compact composition operators
on $\mathcal{B}$ and $H^\infty$(see \cite{MM} and \cite{CM}).

\begin{lem}
  If $\varphi\in S(\mathbb{D})$, then
  \begin{description}
    \item[(1)] Every $\varphi$ induces an bounded composition operator on $\mathcal B$ and $H^\infty$.
    \item[(2)] $C_\varphi:\mathcal{B}\rightarrow\mathcal{B}$ is compact if and only if $C_\varphi$ is bounded and
    $$\lim\limits_{|\varphi(z)|\rightarrow1}|\varphi^\#(z)|=0.$$
    \item[(3)] $C_\varphi:H^\infty\rightarrow H^\infty$ is compact if and only if $\overline{\varphi(\mathbb{D})}\subset\mathbb{D}$.
  \end{description}
\end{lem}

The following criterion for compactness follows from standard
arguments, that the proof is similar to the method of Proposition 3.11 in
\cite{CM}. Hence we omit the details.
\begin{lem}
  Suppose that $\varphi\in S(\mathbb{D})$ and $g\in
  H(\mathbb{D})$. Then $C_\varphi V_g-V_gC_\varphi$
  is compact from $X$ to $\mathcal{B}$ if and only if for any bounded sequence $\{f_k\},\,k=1,2,...$ in $X$
  which converges to zero uniformly on compact subsets of $\mathbb{D}$,
  $\|(C_\varphi V_g-V_gC_\varphi)f_k\|_*\rightarrow0$ as $k\rightarrow\infty$.
\end{lem}

Recall that the notation $\mathbb{C}^N$ represents the $N$
dimensional complex Euclidean space. Denote the unit ball of
$\mathbb{C}^N$ by $B_N$. If $z,w\in B_N$, we define M\"{o}bius transform by
$$\Phi_w(z)=\frac{a-P_a(z)-s_aQ_a(z)}{1-\langle z,a\rangle},$$
where $P_a(z)=\frac{\langle z, a\rangle}{\langle a,a\rangle}a$, $Q_a(z)=z-P_a(z)$ and $s_a=\sqrt{1-|a|^2}$.
 The following Lemma will be used in Section
4, which was first presented by Berndtsson in \cite{B4}.

\begin{lem} Let $\left\{x_{i}\right\}$ be a sequence in the ball
$B_{N}$ satisfying
\begin{equation}\label{14}
\prod_{j:j\neq k}|\Phi_{x_{j}}(x_{k})|\geq d>0 \quad\mbox{for
any}\quad k.
\end{equation}
Then there exists a number $M=M(d)<\infty$ and a sequence of
functions $h_{k}\in H^{\infty}(B_{N})$ such that
\begin{equation}\label{15}
(a)\; \;h_{k}(x_{j})=\delta_{kj}; \quad (b)\;\;
\sum_{k}|h_{k}(z)|\leq M \quad \mbox{for} \quad |z|<1.
\end{equation}
(The symbol $\delta_{kj}$ is equal to $1$ if $k=j$ and $0$
otherwise.)
\end{lem}

The next lemma was proved by Carl Toews in \cite{T}.

\begin{lem} Let $\{z_{n}\}\subset B_{N}$ be a sequence with
$|z_{n}|\rightarrow 1$ as $n\rightarrow \infty$. Then for any given
$d\in(0,1)$ there is a subsequence such that
$\left\{x_{i}\right\}:=\left\{z_{n_{i}}\right\}$ satisfies
\eqref{14}.
\end{lem}

From this lemma, there is always a subsequence which
satisfies \eqref{14} for every sequence converging to the boundary
of $B_N$, and Lemma 2.2 holds for this subsequence.
We just need the result in one dimension.

To get some simple consequences of our main problems, we will
consider the situation in the little Bloch setting. The next lemma is well known, see \cite{OSZ}.
\begin{lem}
  A closed set $K$ in $\mathcal{B}_0$ is compact if and only if it is bounded and
  satisfies $$\lim\limits_{|z|\rightarrow1}\sup\limits_{f\in K}(1-|z|^2)|f'(z)|=0.$$
\end{lem}

In our discussion, we will use the boundedness of operators $J_g$ and $I_g$, from $H^\infty$ or $\mathcal B$ to $\mathcal{B}$.
Several characterizations are listed below.
\begin{lem}\label{l2.6}
  Suppose that $g\in H(\mathbb{D})$.
  Then
  \begin{description}
    \item[(Corollary 4 in \cite{LS})]  $J_g: H^\infty\rightarrow\mathcal{B}$
  is bounded if and only if $g\in\mathcal{B}$;
    \item[(Theorem 3 in \cite{LS7'})]  $J_g: \mathcal{B}\rightarrow\mathcal{B}$
  is bounded if and only if $g\in\mathcal{B}_{\log}$;
    \item[(Corollary 1 in \cite{LS} and Theorem 14 in \cite{LS7'})]  $I_g: H^\infty\mbox{ or }\mathcal B \rightarrow\mathcal{B}$
  is bounded if and only if $g\in H^\infty$.
  \end{description}
\end{lem}

Some definitions and results in Geometric Function Theory are needed, and interested readers can refer to \cite{G1} and \cite{K}.
For $\zeta\in\partial\mathbb{D}$ and $M>1$ the \emph{nontangential approaching region} at $\zeta$ is defined by
$$\Gamma(\zeta,M)=\{z\in\mathbb{D}:|z-\zeta|<M(1-|z|^2)\}.$$ A function $f$ is said to have a \emph{nontangential limit}
at $\zeta$ if $\lim_{z\rightarrow\zeta}f(z)$ exists in each nontangential region $\Gamma(\zeta,M)$, and we denote
it by $\angle-\lim_{z\rightarrow\zeta}f(z)$.
If $\varphi\in S(\mathbb{D})$ and $\zeta\in\partial\mathbb{D}$, we will call $\zeta$ a \emph{boundary fixed point} of $\varphi$
if $$\angle-\lim_{z\rightarrow \zeta}\varphi(z)=\zeta.$$
We say $\varphi$ has a \emph{finite angular derivative} at $\zeta\in\partial\mathbb{D}$ if there is $\eta\in\partial\mathbb{D}$
so that$(\varphi(z)-\eta)/(z-\zeta)$ has finite nontangential limit as $z\rightarrow\zeta$. When it exists as a finite complex
number, this limit is denoted $\varphi'(\zeta)$. A $\varphi\in S(\mathbb{D})$ is said to be \emph{parabolic type}
if $\varphi$ has a boundary fixed point $\zeta$ with $\varphi'(\zeta)=1$. If $\varphi$ is parabolic type,
$\varphi(z)\rightarrow\zeta$ and $\varphi'(z)\rightarrow1$ as $z\rightarrow\zeta$
unrestricted in the unit disk, we say $\zeta$ is a $C^1$ parabolic boundary fixed point of $\varphi$.

\section{The case of intertwining operator $I_g$}

First we consider $C_\varphi I_g-I_g C_\varphi$ as an operator
from the Bloch space to itself.
\begin{thm}\label{t3.1}
  Suppose that $\varphi\in S(\mathbb{D})$ and $g\in
  H(\mathbb{D})$. Then $C_\varphi I_g-I_g C_\varphi$ is bounded on the Bloch space
if and only if
 \begin{equation}\label{1}
 \sup\limits_{z\in\mathbb{D}}|\varphi^{\#}(z)|\cdot|g(\varphi(z))-g(z)|<\infty.
 \end{equation}
\end{thm}

\begin{proof}
  Suppose \eqref{1} holds, we will show that $C_\varphi I_g-I_g C_\varphi$ is bounded on $\mathcal{B}$.
  For any $f\in\mathcal{B}$,
\begin{eqnarray*}
   ((C_\varphi I_g-I_g C_\varphi)f)(z)
    &=&C_\varphi\int_0^zf'(\zeta)g(\zeta)d\zeta-
    I_gf(\varphi(z))\\
    &=&\int_0^{\varphi(z)}f'(\zeta)g(\zeta)d\zeta-\int_0^z(f\circ\varphi)'(\zeta)g(\zeta)d\zeta
    .
   \end{eqnarray*}
   \begin{eqnarray*}
     &&\|(C_\varphi I_g-I_g C_\varphi)f\|_*\\
     &=&\sup\limits_{z\in\mathbb{D}}(1-|z|^2)\left|\varphi'(z)f'(\varphi(z))g(\varphi(z))-
     \varphi'(z)f'(\varphi(z))g(z)\right|\\
     &=&\sup\limits_{z\in\mathbb{D}}(1-|z|^2)\left|\varphi'(z)f'(\varphi(z))(g(\varphi(z))
     -g(z))\right|\\
     &=&\sup\limits_{z\in\mathbb{D}}|\varphi^\#(z)|\left|g(\varphi(z))-g(z)\right|
     (1-|\varphi(z)|^2)
     \left|f'(\varphi(z))\right|\\
     &\leq&\sup\limits_{z\in\mathbb{D}}|\varphi^\#(z)|\left|g(\varphi(z))-g(z)\right|\cdot
     \|f\|_*.
   \end{eqnarray*}
  From which we obtain that $C_\varphi I_g-I_g C_\varphi$ is bounded by \eqref{1}.

   Now suppose that $C_\varphi I_g-I_g C_\varphi$ is bounded on $\mathcal{B}$,
   then there is a constant $C$ such that
   $$\|(C_\varphi I_g-I_gC_\varphi)f\|_*\leq \|(C_\varphi I_g-I_gC_\varphi)f\|_{\mathcal{B}}<C$$
 for $\|f\|_{\mathcal{B}}\leq 1$. We will prove condition
   \eqref{1}. Suppose not, there exists a sequence $\{w_n\}$ in $\mathbb{D}$ such that
   $$\lim\limits_{n\rightarrow\infty}|\varphi^\#(w_n)||g(\varphi(w_n))-g(w_n)|=\infty.$$
   Let $\alpha_n(z)=\frac{\varphi(w_n)-z}{1-\overline{\varphi(w_n)}z}$, and $\alpha_n'(z)
   =-\frac{1-|\varphi(w_n)|^2}{(1-\overline{\varphi(w_n)}z)^2}$ for $n=1, 2, ...$. It is easy
   to check that $\|\alpha_n\|_*=1$.
   \begin{eqnarray*}
     &&\|(C_\varphi I_g-I_gC_\varphi)\alpha_n\|_*\\
     &=&\sup\limits_{z\in\mathbb{D}}|\varphi^\#(z)||g(\varphi(z))-g(z)|\cdot
     (1-|\varphi(z)|^2)|\alpha_n'(\varphi(z))|\\
     &\geq&(1-|\varphi(w_n)|^2)\frac{1-|\varphi(w_n)|^2}{|(1-\overline{\varphi(w_n)}\varphi(w_n))^2|}
     \cdot|\varphi^\#(w_n)||g(\varphi(w_n))-g(w_n)| \\
     &=& |\varphi^\#(w_n)||g(\varphi(w_n))-g(w_n)|\rightarrow\infty.
   \end{eqnarray*}
   That is impossible. So \eqref{1} holds.
\end{proof}

When we investigate essential commutativity of $C_\varphi$ and $I_g$,
we need add the condition $g\in H^\infty$ to ensure the boundedness
of $I_g$ on the Bloch space, see Lemma 2.5.
\begin{thm}\label{t3.2}
  Suppose that $\varphi\in S(\mathbb{D})$ and $g\in H^\infty(\mathbb{D})$.
  Then $C_\varphi$ and $I_g$ are essentially commutative on $\mathcal{B}$ if and only if
  \begin{equation}\label{2}
    \lim\limits_{|\varphi(z)|\rightarrow1}|\varphi^\#(z)||g(\varphi(z))-g(z)|=0.
  \end{equation}
\end{thm}
\begin{proof} \emph{Sufficiency}. Note that $g\in H^\infty$ and $\|\varphi^\#\|_\infty\leq1$, we have
  $$\sup\limits_{z\in\mathbb{D}}|\varphi^{\#}(z)|\cdot|g(\varphi(z))-g(z)|<\infty.$$
 From which we conclude that $C_\varphi I_g-I_gC_\varphi$ is a bounded operator by Theorem \ref{t3.1}.
  For any bounded sequence $\{f_k\}$ in $\mathcal{B}$ converging to zero uniformly
  on compact subsets of $\mathbb{D}$ with $\|f_k\|_{\mathcal{B}}\leq M$. From \eqref{2}, it follows that for any small $\varepsilon>0$,
  there exists a $\delta>0$ such that
$$|\varphi^\#(z)||g(\varphi(z))-g(z)|<\frac{\varepsilon}{M}$$ with $\varphi(z)\in\mathbb{D}\setminus(1-\delta)\mathbb{D}$.
\begin{eqnarray}
    &&\|(C_\varphi I_g-I_gC_\varphi)f_k\|_*\nonumber\\
    &=&\sup\limits_{z\in\mathbb{D}}|\varphi^\#(z)||g(\varphi(z))-g(z)|
    (1-|\varphi(z)|^2)|f_k'(\varphi(z))|\nonumber\\
   \label{3} &\leq &(A) +(B),
    \end{eqnarray}
where
$$(A)=\sup\limits_{\varphi(z)\in(1-\delta)\mathbb{D}}|\varphi^\#(z)||g(\varphi(z))-g(z)|
    (1-|\varphi(z)|^2)|f_k'(\varphi(z))|,$$
    and $$(B)=\sup\limits_{\varphi(z)\in\mathbb{D}\setminus
    (1-\delta)\mathbb{D}}|\varphi^\#(z)||g(\varphi(z))-g(z)|
    (1-|\varphi(z)|^2)|f_k'(\varphi(z))|.$$
   It is obvious that $(A)<\varepsilon$ for sufficient large $k$ , and
    $$(B)<\frac{\varepsilon}{M}\cdot\sup\limits_{z\in\mathbb{D}}
    (1-|\varphi(z)|^2)|f_k'(\varphi(z))|=\frac{\varepsilon}{M}\|f_k\|_*<\varepsilon.$$
Thus $C_\varphi I_g-I_gC_\varphi$ is compact on $\mathcal{B}$.

 \emph{Necessity}. If $C_\varphi I_g-I_gC_\varphi$ is compact on $\mathcal{B}$,
  certainly $C_\varphi I_g-I_gC_\varphi$ is a bounded operator. Therefore
  $$\sup\limits_{z\in\mathbb{D}}|\varphi^{\#}(z)|\cdot|g(\varphi(z))-g(z)|<\infty$$
  by Theorem \ref{t3.1}.
  Suppose the
  $$\sup\limits_{|\varphi(z)|\to1}|\varphi^\#(z)||g(\varphi(z))-g(z)|\neq0,$$ then there is a sequence $\{w_n\}$ in $\mathbb D$ with
  $|\varphi(w_n)|\to1$ and an $\epsilon>0$ such that
  \begin{equation}\label{a}|\varphi^\#(w_n)||g(\varphi(w_n))-g(w_n)|>\epsilon.\end{equation}
  Let \begin{equation}\label{4}h_n(z)=\frac{1-|\varphi(w_n)|^2}{1-\overline{\varphi(w_n)}z}.\end{equation}
  A simple computation shows that
  $$h_n'(z)=(1-|\varphi(w_n)|^2)\frac{\overline{\varphi(w_n)}}{(1-\overline{\varphi(w_n)}z)^2}.$$
  So $\|h_n\|_*\leq1$ and $\{h_n\}$ converges to zero uniformly on compact subsets
  of $\mathbb{D}$. By Lemma 2.2, it follows that $\|(C_\varphi
  I_g-I_gC_\varphi)h_n\|_*\rightarrow0$. On the other hand, by
  \eqref{a}, we have
\begin{eqnarray*}
    &&\|(C_\varphi I_g-I_gC_\varphi)h_n\|_*\\
    &=&\sup\limits_{z\in\mathbb{D}}|\varphi^\#(z)|
    |g(\varphi(z))-g(z)|\cdot(1-|\varphi(z)|^2)|h_n'(\varphi(z))|\\
    &=&\sup\limits_{z\in\mathbb{D}}|\varphi^\#(z)|
    |g(\varphi(z))-g(z)|\cdot(1-|\varphi(z)|^2)
    \frac{|\varphi(w_n)|(1-|\varphi(w_n)|^2)}
    {|1-\overline{\varphi(w_n)}\varphi(z)|^2}\\
    &\geq&|\varphi^\#(w_n)|
    |g(\varphi(w_n))-g(w_n)|
    |\varphi(w_n)|>\epsilon
  \end{eqnarray*}
  when $n\rightarrow\infty$. And we find a contradiction. So
  \eqref{2} holds when $C_\varphi$ and $I_g$ essentially commute.
\end{proof}

\begin{cor}
    Let $\varphi\in S(\mathbb{D})$ and
    $g\in H(\mathbb{D})$, then $C_\varphi I_g-I_g C_\varphi: H^\infty
    \rightarrow\mathcal{B}$ is bounded if and only if \eqref{1}
    holds; $C_\varphi I_g-I_g C_\varphi: H^\infty
    \rightarrow\mathcal{B}$ is compact if and only if \eqref{2}
    holds.
\end{cor}
\begin{proof} To prove this corollary, we only adjust \eqref{4}, in the proof of Theorem 3.2, as
  $$f_n(z)=\frac{1-|\varphi(w_n)|^2}{1-\overline{\varphi(w_n)}z}
  \cdot\frac{\varphi(w_n)-z}{1-\overline{\varphi(w_n)}z}.$$ The other proofs are similar to that of
  Theorems \ref{t3.1} and \ref{t3.2}. We omit the details.
\end{proof}
Using Lemma 2.4, we can easily get the following corollary in the little
Bloch setting. The method is as before and we omit its proof.
\begin{cor}
  Let $\varphi\in S(\mathbb{D})$ and $g\in H(\mathbb{D})$, and let $X$ represent
  the space $H^\infty$, $\mathcal{B}$ or $\mathcal{B}_0$.
  The following three statements are equivalent:
  \begin{description}
    \item[(a)] $C_\varphi I_g-I_gC_\varphi : X\rightarrow\mathcal{B}_0$ is bounded;
    \item[(b)] $C_\varphi I_g-I_gC_\varphi : X\rightarrow\mathcal{B}_0$ is compact;
    \item[(c)] $\lim\limits_{|z|\rightarrow1}|\varphi^{\#}(z)||g(\varphi(z))-g(z)|=0$.
  \end{description}
\end{cor}
Now we consider \textbf{(Q1)} raised in the first section:
\begin{itemize}
  \item When does
$C\mid_{X}\vp_K C\mid_{\mathcal{B}}$ $(V_g)$ hold?
\end{itemize}
\begin{thm} $C\mid_{X}\vp_K C\mid_{\mathcal{B}}$ $(V_g)$ holds 
if and only if $g$ is a constant.
\end{thm}
\begin{proof}
  Sufficiency is obvious. To verify the necessity, just consider $\varphi$ as any
  automorphism in condition \eqref{2}. Then maximum modulus theorem implies that
  $g$ must be a constant.
\end{proof}

We have following result which answers \textbf{(Q2)} in part.

\begin{prop}
  Suppose $\varphi\in S(\mathbb{D})$ and $g\in H^\infty$ and let $X$ denote either $\mathcal{B}$ or $H^\infty$.
If $\varphi$ satisfies
    \begin{equation}
  \label{27}\lim\limits_{|\varphi(z)|\rightarrow1}\frac{|\varphi(z)-z|}
  {(1-\max\{|z|,|\varphi(z)|\})^2}\left|\varphi^\#(z)\right|=0,
\end{equation}
    then $C_\varphi\mid_{X}\vp_K
  C_\varphi\mid_{\mathcal{B}}\quad(I\mid_{X\rightarrow\mathcal{B}})$.

\end{prop}

\begin{proof}
  By the Cauchy integral formula, one finds that
  $$|g'(z)|\leq\frac{\|g\|_\infty}{(1-|z|)^2}\quad(\forall z\in\mathbb{D})$$ for $g\in H^\infty$.
  The proposition will be proved immediately by noting that
  \begin{eqnarray*}
    &&|\varphi^\#(z)||g(\varphi(z))-g(z)|\\&=&|\varphi^\#(z)|\left|\int_z^{\varphi(z)} g'(\zeta)\mathrm{d}\zeta\right|\\
    &\leq&|\varphi^\#(z)|\frac{\|g\|_\infty}{(1-\max{\{|\varphi(z)|,|z|\}})^2}\int_z^{\varphi(z)}\left|\mathrm{d}\zeta\right|\\
    &\leq&|\varphi^\#(z)|\|g\|_\infty\frac{|\varphi(z)-z|}{(1-\max{\{|\varphi(z)|,|z|\}})^2}
  \end{eqnarray*}
\end{proof}

According to
Lemma \ref{l2.6}, 
the operator $I_g: X\rightarrow \mathcal B$ is bounded if $g$ is analytic in $\mathbb D$ and continuous to the boundary.
The next theorem will answer \textbf{(Q2)} in part. We will get a necessary and sufficient condition of the compact intertwining relation,
when $g$ is the function in the disk algebra.
\begin{thm}
  Let $\mathcal{A}$ be the disk algebra, that is the subspace of $H^\infty$ whose elements are
analytic in $\mathbb{D}$ and continuous to the unit circle. Then
$$C_\varphi\mid_{X}\vp_K C_\varphi\mid_{\mathcal{B}}\quad(I_g\mid_{X\rightarrow\mathcal{B}})$$
holds for all $g\in\mathcal{A}$ if and only if
\begin{equation}
  \label{26}\lim\limits_{|\varphi(z)|\rightarrow1}|\varphi^\#(z)||\varphi(z)-z|=0.
\end{equation}
\end{thm}
\begin{proof} Necessity can be proved immediately by putting $g=\mathrm{id}$ in equation \eqref{2}, we just
  prove sufficiency. Condition \eqref{26} is sufficient to make
  $$C_\varphi\mid_{X}\vp_K C_\varphi\mid_{\mathcal{B}}\quad(I_g)$$ hold for each
  monomial $g(z)=z^n$. For each $h\in \mathcal{A}$, there is a polynomial sequence $\{g_n\}$ such that
  $g_n\rightarrow h$ in supreme norm. One has
\begin{eqnarray*}
  &&\|C_\varphi I_h-I_h C_\varphi\|_{e, X\rightarrow\mathcal{B}}\\&=&\inf
\left\{\|C_\varphi I_h-I_h C_\varphi-K\|_{X\rightarrow\mathcal{B}}:\;K\in\mathcal{K}(X, \mathcal{B})\right\}\\
&\leq&\|C_\varphi I_h-I_h C_\varphi-(C_\varphi I_{g_n}-I_{g_n}C_\varphi)\|_{X\rightarrow\mathcal{B}}\\
&\leq&2\|C_\varphi\|_{X\rightarrow\mathcal{B}}\cdot\|I_h-I_{g_n}\|_{X\rightarrow\mathcal{B}}\\
&=&2\|C_\varphi\|_{X\rightarrow\mathcal{B}}\cdot\sup_{\|f\|_X\leq1}\left\|\int_0^z f'(\zeta)
(h-g_n)(\zeta)\d\zeta\right\|_{\mathcal{B}}\\
&\leq&2\|C_\varphi\|_{X\rightarrow\mathcal{B}}\cdot\|h-g_n\|_\infty\cdot\sup_{\|f\|_X\leq1}
\left\|\int_0^z f'(\zeta)\d\zeta\right\|_{\mathcal{B}}\\
&=&2\|C_\varphi\|_{X\rightarrow\mathcal{B}}\cdot\|h-g_n\|_\infty
\end{eqnarray*}
Hence $\|C_\varphi I_h-I_h C_\varphi\|_{e, X\rightarrow\mathcal{B}}=0$ by taking $n\rightarrow\infty$.
\end{proof}

\section{The case of intertwining operator $J_g$}

First, we consider the case for $X=H^\infty$ in the main question.
\begin{thm}
  Suppose that $\varphi\in S(\mathbb{D})$ and $g\in H(\mathbb{D})$.
  Then
  \begin{enumerate}
    \item $C_\varphi J_g-J_gC_\varphi$ is bounded from $H^\infty$ to $\mathcal{B}$
    if and only if
    \begin{equation}\label{5}
      \sup\limits_{z\in\mathbb{D}}(1-|z|^2)|(g\circ\varphi)'(z)-g'(z)|<\infty;
    \end{equation}
    \item $C_\varphi J_g-J_gC_\varphi$ is compact from $H^\infty$ to $\mathcal{B}$
    if and only if \eqref{5} holds and
    \begin{equation}\label{6}
      \lim\limits_{|\varphi(z)|\rightarrow1}(1-|z|^2)|(g\circ\varphi)'(z)-g'(z)|=0.
    \end{equation}
  \end{enumerate}
\end{thm}

\begin{proof}
  Being similar to the proofs of Theorems 3.1 and 3.2, we just need some modifications.
  \begin{equation}\label{7}
    \begin{array}{ll}
    &\|(C_\varphi J_g-J_gC_\varphi)f\|_*\\
    =&\left\|\int^{\varphi(z)}_0f(\zeta)g'(\zeta)d\zeta
    -\int^z_0f(\varphi(\zeta))g'(\zeta)d\zeta\right\|_*\\
    =&\sup\limits_{z\in\mathbb{D}}(1-|z|^2)|f(\varphi(z))g'(\varphi(z))\varphi'(z)
    -f(\varphi(z))g'(z)|\\
    =&\sup\limits_{z\in\mathbb{D}}(1-|z|^2)|(g\circ\varphi)'(z)-g'(z)||f(\varphi(z))|.
    \end{array}
  \end{equation}

  Sufficiency of the two items in the theorem is obvious from the last formula in \eqref{7}.

  Necessity of boundedness can be proved by computing test functions
  $$f_n(z)=1-\frac{\varphi(w_n)-z}{1-\overline{\varphi(w_n)}z},$$
  where sequence $\{w_n\}$ violates equation \eqref{5}.

  Necessity of compactness. Suppose not, we can find a sequence $\{\varphi(w_n)\}$
  converging to the boundary of $\mathbb{D}$ and $\epsilon_0>0$ such that
  \begin{equation}
    \label{10}\lim\limits_{n\rightarrow\infty}(1-|w_n|^2)|(g\circ\varphi)'(w_n)-g'(w_n)|>\epsilon_0.
  \end{equation}
  Further we may assume that $\{\varphi(w_n)\}$ is interpolating. Then there exist
  functions $\{h_n\}$ in $H^\infty$ for $\{\varphi(w_n)\}$ such that
  \begin{equation}\label{8}
    h_n(\varphi(w_k))=\left\{\begin{array}
      {ll}1\quad n=k,\\ 0\quad n\neq k.
    \end{array}\right.
  \end{equation}
  and \begin{equation}\label{9}
    \sum\limits_n|h_n(z)|\leq M<\infty,
  \end{equation}
  by Lemma 2.3 and 2.2, or see \cite{G1}. Equation \eqref{9} guarantees that $\{h_n\}$ is bounded in $H^\infty$ and
  converges to zero on compact subsets of $\mathbb{D}$.
  \begin{eqnarray*}
    &&\|(C_\varphi J_g-J_gC_\varphi)h_n\|_*\\
    &=&\sup\limits_{z\in\mathbb{D}}(1-|z|^2)|(g\circ\varphi'(z)-g'(z)||h_n(\varphi(z))|\\
    &\geq&(1-|w_n|^2)|(g\circ\varphi'(w_n)-g'(w_n)||h_n(\varphi(w_n))|\\
    &=&(1-|w_n|^2)|(g\circ\varphi)'(w_n)-g'(w_n)|.
  \end{eqnarray*}
  The last equation follows by \eqref{8}. Letting $n\rightarrow\infty$, we find contradiction
  by \eqref{10}.
\end{proof}

\begin{cor}
  Let $\varphi\in S(\mathbb{D})$ and $g\in H(\mathbb{D})$. The
  three following conditions are equivalent:
  \begin{description}
    \item[(a)] $C_\varphi J_g-J_gC_\varphi :H^\infty\rightarrow\mathcal{B}_0$ is bounded;
    \item[(b)] $C_\varphi J_g-J_gC_\varphi :H^\infty\rightarrow\mathcal{B}_0$ is compact;
    \item[(c)] $\lim\limits_{|z|\rightarrow1}(1-|z|^2)|(g\circ\varphi)'(z)-g'(z)|=0$.
  \end{description}
\end{cor}

 For composition operator and $J_g$, the result of \textbf{(Q1)} turns out to be
 interesting. Note that $g\in\mathcal{B}$ implies that $\sup\limits_{z\in\mathbb{D}}(1-|z|^2)|
(g\circ\varphi)'(z)-g'(z)|<\infty$, thus $C_\varphi J_g-J_gC_\varphi$ is a
bounded operator from $H^\infty$ to $\mathcal{B}$. And
$$C_\varphi\mid_{H^\infty}\vp_K C_\varphi\mid_{\mathcal{B}}\quad(J_g\mid_{H^\infty\rightarrow\mathcal{B}})$$
if and only if \eqref{6} holds. Now, we can answer \textbf{(Q1)}.
\begin{cor}
  If $g\in \mathcal{B}_0$, then
  \begin{equation}\label{11}
    C\mid_{H^\infty}\vp_K C\mid_{\mathcal{B}}\quad(J_g\mid_{H^\infty\rightarrow\mathcal{B}}).
  \end{equation} 
\end{cor}
\begin{proof}
  Since $g$ is in the little Bloch space, $(1-|z|^2)|g'(z)|$ tends to
  $0$ whenever $z$ tends to the boundary of the disk. So we have that
  \begin{eqnarray*}
    &&\lim\limits_{|\varphi(z)|\rightarrow1}(1-|z|^2)|g'(\varphi(z))\varphi'(z)-g'(z)|\\
    &\leq&\lim\limits_{|\varphi(z)|\rightarrow1}|\varphi^\#(z)|(1-|\varphi(z)|^2)|g'(\varphi(z))|
    +\lim\limits_{|\varphi(z)|\rightarrow1}(1-|z|^2)|g'(z)|.
  \end{eqnarray*}
  Conditions $g\in\mathcal{B}$ and $|\varphi^\#|\leq1$ imply that $C_\varphi J_g-J_gC_\varphi$ is
  compact from $H^\infty$ to $\mathcal{B}$ for every self-mapping $\varphi$.
\end{proof}
In contrast to composition operators, $C_\varphi$, and operators of the form $I_g$, there are many non-constant functions g,
which are actually the little Bloch functions, such that $J_g$ essentially commutes with all the composition operators.
Naturally, we are going to ask:
Does $J_g$ commuting essentially with all the composition operators
imply that $g$ is in the little Bloch space? The answer is positive.
The next lemma is the key lemma to study \textbf{(Q1)}.
The method of proof is the same as in our recent papers \cite{TZ2,TZ}.

\begin{lem}\label{l1}
  If $g$ is a Bloch function on the unit disk with the property that, for any rotation $\tau(z)=e^{\textbf{i}t}z$,
  $g\circ\tau-g$ is in the little Bloch space, then $g$ itself must be in the little Bloch space.
\end{lem}
\begin{proof}
  Condition $g\in\mathcal{B}$ is necessary to ensure that $J_g: H^\infty\rightarrow\mathcal{B}$ is bounded.
  Since $g\circ\tau-g$ is in the little Bloch space,
  we have
  \begin{equation}\label{12}
    \lim\limits_{|z|\rightarrow1}(1-|z|^2)\left|g'(e^{\textbf{i}t}z)e^{\textbf{i}t}-g'(z)\right|=0.
  \end{equation}
  It is necessary to estimate the upper bound of left formula in \eqref{12}.
  \begin{eqnarray*}
    &&(1-|z|^2)\left|g'(e^{\textbf{i}t}z)e^{\textbf{i}t}-g'(z)\right|\\
    &\leq& (1-|z|^2)\left|g'(e^{\textbf{i}t}z)e^{\textbf{i}t}\right|
    +(1-|z|^2)\left|g'(z)\right|\\
    &=&(1-|e^{\textbf{i}t}z|^2)\left|g'({e^{\textbf{i}t}}z)\right|
    +(1-|z|^2)\left|g'(z)\right|\\
    &\leq& 2\|g\|_*
  \end{eqnarray*}
  Thus the left formula in equation \eqref{12} is finite independent of $t$, since $g\in\mathcal{B}$.

  Suppose that $g(z)=\sum\limits_{n=0}^\infty a_nz^n,$ then $g'(z)=
  \sum\limits_{n=1}^\infty na_nz^{n-1}.$  Integrating with respect to $t$ from $0$ to $2\pi$, and some calculations
  show that
  \begin{eqnarray*}
    &&\int_0^{2\pi}\lim\limits_{|z|\rightarrow1}(1-|z|^2)\left|g'(e^{\textbf{i}t}z)e^{\textbf{i}t}-g'(z)\right|\textrm{d}t\\
    &=&\lim\limits_{|z|\rightarrow1}(1-|z|^2)\left|\int_0^{2\pi}g'(e^{\textbf{i}t}z)e^{\textbf{i}t}-g'(z)
    \textrm{d}t\right|\\
    &=&\lim\limits_{|z|\rightarrow1}(1-|z|^2)\left|\int_0^{2\pi}\sum\limits_{n=1}^\infty
    \left(na_n(e^{\textbf{i}t}z)^{n-1}e^{\textbf{i}t}-na_nz^{n-1}\right)\textrm{d}t\right|\\
    &=&\lim\limits_{|z|\rightarrow1}(1-|z|^2)\left|\sum\limits_{n=1}^\infty na_nz^{n-1}
    \int_0^{2\pi}(e^{\textbf{i}nt}-1)\textrm{d}t\right|\\
    &=&2\pi\cdot\lim\limits_{|z|\rightarrow1}(1-|z|^2)|g'(z)|,
  \end{eqnarray*}
  where the equation in the second line is true by the dominated convergence theorem.
  Thus we get $g\in\mathcal{B}_0$ from \eqref{12}.
\end{proof}
\begin{thm}\label{t4.5}
  $ C\mid_{H^\infty}\vp_K C\mid_{\mathcal{B}}$ $(J_g\mid_{H^\infty\rightarrow\mathcal{B}})$
  if and only if $g\in\mathcal{B}_0$.
\end{thm}
\begin{proof}
  Sufficiency is stated in Corollary 4.3. To prove necessity, just consider the rotation
  of the disk. That is to say, by putting $\varphi(z)=\tau(z)=e^{\textbf{i}t}z$ in the condition
  \eqref{6}, and we have
  $g\in\mathcal{B}_0$ by Lemma \ref{l1}.
\end{proof}

Considering the composition operators $C_\varphi$ and Volterra operator $J_g$, we can obtain a necessary
condition (Proposition \ref{p4.2}) and a sufficient condition (Proposition \ref{p4.3}) to answer \textbf {(Q2)} partially.

\begin{prop}\label{p4.2}
  Let $\zeta\in\partial\mathbb{D}$, $\varphi\in S(\mathbb{D})$ and $\varphi(z)$ tend to the unit circle when $z$ converges to $\zeta$
nontangentially. If
  $$C_\varphi\mid_{H^\infty}\vp_K C_\varphi\mid_{\mathcal{B}}\quad
  (J\mid_{H^\infty\rightarrow\mathcal{B}}),$$ then
  $\angle-\lim\limits_{z\rightarrow\zeta}\varphi(z)=\zeta$. 
\end{prop}


\begin{proof}
  By Theorem \ref{t4.5}, $\lim\limits_{|\varphi(z)|\rightarrow1}(1-|z|^2)|(g\circ\varphi-g)'(z)|=0$
  holds for every $g\in\mathcal{B}$. Suppose that $\varphi(z)\rightarrow\partial\mathbb{D}$
  as $z\rightarrow\zeta$ in some nontangential approaching region, 
such that $\eta\neq\zeta$ and $\varphi(z_n)\nrightarrow\eta$.
  let $g(z)=-\log(1-\bar{\zeta}z)$ in equation \eqref{11}, then
  \begin{eqnarray*}
    &&\lim\limits_{z\rightarrow\zeta}(1-|z|^2)\left|\frac{\bar{\zeta}\varphi'(z)}{1-\bar{\zeta}\varphi(z)}
  -\frac{\bar{\zeta}}{1-\bar{\zeta}z}\right|\\
  &=&\lim\limits_{z\rightarrow\zeta}\frac{1-|z|^2}{|1-\bar{\zeta}z|}
  \left|\frac{\zeta-z}{\zeta-\varphi(z)}\varphi'(z)-1\right|=0.
  \end{eqnarray*}
  For any $M>0$, if $z\in\Gamma(\zeta,M)=\{z\in\mathbb{D}:
  |\zeta-z|<M(1-|z|^2)\}$, $$\angle-\lim\limits_{z\rightarrow\zeta}
  \left|\frac{\zeta-z}{\zeta-\varphi(z)}\varphi'(z)-1\right|=0.$$
  By noting that $|\zeta-z||\varphi'(z)|=\frac{|\zeta-z|}{1-|z|^2}\cdot(1-|\varphi(z)|^2)|\varphi^\#(z)|\rightarrow0$,
  one has \begin{equation}\label{24}\angle-\lim\limits_{z\rightarrow\zeta}\varphi(z)=\zeta.\end{equation}

%

\end{proof}

\begin{prop}\label{p4.3}
  Let $\varphi\in S(\mathbb{D})$ be of parabolic type with $C^1$ boundary fixed point 
$\zeta\in\partial\mathbb{D}$. If
\begin{equation}
  \label{28}\sup\limits_{z\in\mathbb{D}}\frac{|\varphi(z)-z|}{\left(1-\max{\{|z|,|\varphi(z)|\}}\right)^2}\leq\infty,
\end{equation}
   then $C_\varphi\mid_{H^\infty}\vp_K C_\varphi\mid_{\mathcal{B}}$ $(J\mid_{H^\infty\rightarrow\mathcal{B}}).$
\end{prop}

\begin{proof}
  It is well known that $g\in \mathcal{B}$ if and only if $$\sup_{z\in\mathbb{D}}(1-|z|^2)^n|g^{(n)}(z)|<\infty.$$
Note that
\begin{eqnarray*}
  &&(1-|z|^2)|(g\circ\varphi)'(z)-g'(z)|\\
&\leq&(1-|z|^2)|\varphi'(z)||g'(\varphi(z))-g'(z)|+(1-|z|^2)|g'(z)||\varphi'(z)-1|\\
&=&(1-|z|^2)|\varphi'(z)|\left|\int_z^{\varphi(z)}g''(\zeta)\mathrm{d}\zeta\right|+(1-|z|^2)|g'(z)||\varphi'(z)-1|\\
&\leq&(1-|z|^2)|\varphi'(z)|\sup\limits_{z\in\mathbb{D}}(1-|z|^2)^2|g''(z)|\int_z^{\varphi(z)}\frac{|\mathrm{d}\zeta|}
  {(1-|\zeta|)^2}\\&&+(1-|z|^2)|g'(z)||\varphi'(z)-1|\\
&\leq&(1-|z|^2)|\varphi'(z)|\sup\limits_{z\in\mathbb{D}}(1-|z|^2)^2|g''(z)|\frac{|\varphi(z)-z|}
  {(1-\max{\{|z|,|\varphi(z)|\}})^2}\\&&+(1-|z|^2)|g'(z)||\varphi'(z)-1|
\end{eqnarray*}
where the integral path is chosen to be the segment from $z$ to $\varphi(z)$.
  Then $C_\varphi\mid_{H^\infty}\vp_K C_\varphi\mid_{\mathcal{B}}$ $(J\mid_{H^\infty\rightarrow\mathcal{B}})$
follows immediately by those conditions in the proposition.
\end{proof}

The next several propositions concern $C_\varphi$ and $J_g$ as maps from $\mathcal{B}$ to itself.
\begin{prop}
  Assume that $\varphi\in S(\mathbb{D})$ and $g\in H(\mathbb{D})$. Then
  $C_\varphi J_g-J_gC_\varphi$ is bounded from $\mathcal{B}$ to itself if
  and only if
  \begin{equation}
    \label{17}\sup\limits_{z\in\mathbb{D}}(1-|z|^2)|(g\circ\varphi)'(z)-g'(z)|\log\frac{2}{1-|\varphi(z)|^2}<\infty.
  \end{equation}
\end{prop}
\begin{proof}
  Sufficiency can be verified by some straightforward computations and
  inequalities. Necessity will be proved by choosing the test function
  $$f_w(z)=\log\frac{2}{1-\overline{w}z}.$$
\end{proof}

\begin{prop}\label{p4.1}
  Assume that $\varphi\in S(\mathbb{D})$ and $g\in H(\mathbb{D})$. Then the following
  statements are equivalent:
  \begin{description}
    \item[(a)] $C_\varphi J_g-J_gC_\varphi :\mathcal{B}\rightarrow\mathcal{B}$ is
    compact and \eqref{17} holds;
    \item[(b)] $C_\varphi J_g-J_gC_\varphi :\mathcal{B}_0\rightarrow\mathcal{B}_0$ is
    compact;
    \item[(c)] $C_\varphi J_g-J_gC_\varphi :\mathcal{B}_0\rightarrow\mathcal{B}_0$ is
    weakly compact;
    \item[(d)] Condition \eqref{17} holds and
    \begin{equation}\label{18}
      \lim\limits_{|\varphi(z)|\rightarrow1}(1-|z|^2)|(g\circ\varphi)'(z)-g'(z)|\log\frac{2}{1-|\varphi(z)|^2}=0;
    \end{equation}
    \item[(e)] $C_\varphi J_g-J_gC_\varphi :\mathcal{B}\rightarrow\mathcal{B}_0$ is
    compact;
    \item[(f)] $C_\varphi J_g-J_gC_\varphi :\mathcal{B}\rightarrow\mathcal{B}_0$ is
    bounded;
  \end{description}
\end{prop}
The compactness of operators $C_\varphi J_g$ and $J_gC_\varphi$ are
characterized in Theorem 4 and Theorem 11 in \cite{LS7'}. The proofs
are similar to those in Theorem 4 in \cite{LS7'}, so we omit them.

To continue our discussion, we need an upper bound on the modulus of $\varphi(z)$ (see Corollary 2.40
in \cite{CM}).
\begin{lem}\label{l4.1}
  If $\varphi\in S(\mathbb{D})$, then
  $|\varphi(z)|\leq\frac{|z|+|\varphi(0)|}{1+|z||\varphi(0)|}.$
\end{lem}
Now we are ready to consider \textbf{(Q1)} for $C_\varphi$ and $J_g$, where
both of them are operators acting on $\mathcal{B}$.
\begin{thm}
  Assume $g\in H(\mathbb{D})$ is such that $J_g$ is bounded on the Bloch space. Then
$$C\mid_{\mathcal{B}}\vp_K C\mid_{\mathcal{B}}\quad(J_g\mid_{H^\infty\rightarrow\mathcal{B}})$$
  if and only if $g\in\mathcal{B}_{\log,0}$, that is
  \begin{equation}
    \label{20}\lim\limits_{|z|\rightarrow1}(1-|z|^2)|g'(z)|\log\frac{2}{1-|z|^2}=0.
  \end{equation}
\end{thm}
\begin{proof}
  Following the method in the proof of Lemma \ref{l4.1}, the necessity can be proved similarly.

  Now suppose \eqref{20} holds. We have that
  \begin{eqnarray*}
    &&(1-|z|^2)|(g\circ\varphi)'(z)-g'(z)|\log\frac{2}{1-|\varphi(z)|^2}\\
    &\leq&|\varphi^\#(z)|(1-|\varphi(z)|^2)|g'(\varphi(z)|\log\frac{2}{1-|\varphi(z)|^2}\\
    &&+(1-|z|^2)|g'(z)|\log\frac{2}{1-|z|^2}\cdot\frac{\log2-\log(1-|\varphi(z)|^2)}{\log2-\log(1-|z|^2)}.
  \end{eqnarray*}
  Just consider those $z$'s tending to $\partial\mathbb{D}$ with $|\varphi(z)|\rightarrow1$.
  \begin{eqnarray*}
    &&\lim\limits_{|z|\rightarrow1^-}\frac{\log2-\log(1-|\varphi(z)|^2)}{\log2-\log(1-|z|^2)}\\
    &=&\lim\limits_{|z|\rightarrow1^-}\frac{-\log(1-|\varphi(z)|)}{-\log(1-|z|)}\\
    &\leq&\lim\limits_{|z|\rightarrow1^-}\frac{-\log(1-\frac{|z|+|\varphi(0)|}{1+|z||\varphi(0)|})}{-\log(1-|z|)}\\
    &=&\lim\limits_{|z|\rightarrow1^-}\frac{\log(1-|z|)(1-|\varphi(0)|)-\log(1+|z||\varphi(0)|)}{\log(1-|z|)}\\
    &=&\lim\limits_{|z|\rightarrow1^-}\frac{-\frac{1}{1-|z|}-\frac{|\varphi(0)|}{1+|z||\varphi(0)|}}
    {-\frac{1}{1-|z|}}=1
  \end{eqnarray*}
  where Lemma 4.8 and L'Hospital Law are applied. Thus \eqref{18} holds for
  every $\varphi\in S(\mathbb{D})$ by \eqref{20}. The proof is completed by Proposition \ref{p4.1}.
\end{proof}

By the same method as in Proposition \ref{p4.2}, ones can obtain a necessary condition for \textbf{(Q2)} when $X=\mathcal{B}$.
Hence we omit the proof.
\begin{prop}
  Let $\zeta\in\partial\mathbb{D}$, $\varphi\in S(\mathbb{D})$ and $\varphi(z)$ tend to the unit circle when $z$ converges to $\zeta$
nontangentially. If
  $$C_\varphi\mid_{\mathcal{B}}\vp_K C_\varphi\mid_{\mathcal{B}}\quad
  (J\mid_{\mathcal{B}\rightarrow\mathcal{B}}),$$ then
  $\angle-\lim\limits_{z\rightarrow\zeta}\varphi(z)=\zeta$. 
\end{prop}

\begin{prop}
    Let $\varphi\in S(\mathbb{D})$ be parabolic type with $C^1$ boundary fixed point 
$\zeta\in\partial\mathbb{D}$. If
\begin{equation}
  \label{28}\sup\limits_{z\in\mathbb{D}}\frac{|\varphi(z)-z|}{1-\max{\{|z|,|\varphi(z)|\}}}\leq\infty,
\end{equation}
   then $C_\varphi\mid_{\mathcal{B}}\vp_K C_\varphi\mid_{\mathcal{B}}$ $(J\mid_{\mathcal{B}}).$
\end{prop}
\begin{proof}
  The proposition follows by $\mathcal{B}_{\log}\subset\mathcal{B}$. In fact,
\begin{eqnarray*}
  &&\log\frac{2}{1-|\varphi(z)|^2}\approx\log\frac{2}{1-|\varphi(z)|}\leq\log\frac{2(1+|\varphi(z)||z|)}{(1-|z|)(1-|\varphi(z)|)}\\
&\lesssim&\log\frac{2}{1-|z|}\approx\log\frac{2}{1-|z|^2}
\end{eqnarray*}
for $z$ close enough to the unit circle.
By estimating that
\begin{eqnarray*}
  &&(1-|z|^2)|(g\circ\varphi)'(z)-g'(z)|\log\frac{2}{1-|\varphi(z)|^2}\\
&\lesssim&(1-|z|^2)|(g\circ\varphi)'(z)-g'(z)|\log\frac{2}{1-|z|^2}\\
&\leq&(1-|z|^2)\log\frac{2}{1-|z|^2}|\varphi'(z)||g'(\varphi(z))-g'(z)|\\
&&+(1-|z|^2)\log\frac{2}{1-|z|^2}|g'(z)||\varphi'(z)-1|,
\end{eqnarray*}
the proof will be completed as we did in Proposition \ref{p4.3}.
\end{proof}

\textbf{Remark.} Question 1 concerns the subclass of the bounded Volterra operators, whose elements'
essential commutants contain all the composition operators. Question 2 concerns the subclass
of the composition operators, whose elements' essential commutants contain all the bounded
Volterra operators. Answers to \textbf{(Q1)} are complete, but we can only
find some sufficient or necessary conditions for \textbf{(Q2)}. It seems very difficult to
answer \textbf{(Q2)} completely, since the boundary behavior of a
function either in $H^\infty$ or $\mathcal{B}$ can be rather wild.

{\bf Acknowledgement.} 
  The work received great help from Professor Kehe Zhu and Jie Xiao. We would like to take this opportunity
  to express our gratitude.

  We also would like to thank the referee for careful reading and helpful suggestions which improved the presentation.




\begin{thebibliography}{99}

%
%
  \bibitem{AC} A. Aleman and J. A. Cima, An integral operator on $H^p$ and Hardy's
  inequality, J. Anal. Math., Vol. 85, pp. 157-176, 2001.

  \bibitem{AS} A. Aleman and A. G. Siskakis, Integration operators on Bergman spaces,
  Indiana Univ. Math. J., Vol. 46, No. 2, pp. 337-356, 1997.

\bibitem{B4} B. Berndtsson, Interpolating sequences for $H^{\infty}$ in the ball,
Math. Indag. 47 (1985), 1-10; Proc. Kon. Nederl. Akad. Wetens., vol.
88A, pp. 1-10, 1985.
%
%
\bibitem{CM} C. C. Cowen and B. D. MacCluer, Composition Operators on Spaces
of Analytic Functions, CRC Press, Boca Raton, 1995.
%
%
%
%
%
%
%
%
%
%
%
%




%
%
%
%
%
\bibitem{G1} J. B. Garnett, Bounded Analytic Functions, Academic Press, New York, 1981.

\bibitem{K} S. G. Krantz, Geometric function theory: explorations in complex analysis, Birkh\"{a}ser, Boston, 2006.

\bibitem{LS} S. Li and S. Stevi\'{c}, Products of composition and integral type operators
from $H^\infty$ to the Bloch space, Complex Variables and Elliptic
Equations, vol. 53, No. 5, pp. 463-474, 2008.
\bibitem{LS6} S. Li and S. Stevi\'c, Products of Volterra type operator and
composition operator from $H^{\infty}$ and Bloch spaces to the
Zygmund space, J. Math. Anal. Appl. vol. 345, pp. 40-52, 2008.

\bibitem{LS7'} S. Li and S. Stevi\'c, Products of integral-type operators and composition operators
between Bloch-type spaces, J. Math. Anal. Appl. vol. 349, pp.
596-610, 2009.

\bibitem{MM} K. Madigan and A. Matheson, Compact composition operators on the Bloch space, Trans.
Amer. Math. Soc., vol. 347, pp. 2679-2687, 1995.

\bibitem{OSZ} S. Ohno, K. Stroethoff and R. Zhao, Weighted composition
operators between Bloch-type spaces, Rochy Mountain J. Math., vol 33(1),
pp. 191-215, 2003.

\bibitem{S1} J. Shapiro, The essential norm of a composition operator, Ann. Math., vol 125, pp. 375-404, 1987.

\bibitem{Sha} J. H. Shapiro, Composition
operators and classical function theory, Spriger-Verlag, 1993.

  \bibitem{SZ} A. G. Siskakis and R. Zhao, A Volterra type operator on spaces
  of analytic functions. Function spaces (Edwardsville, IL, 1998). Contemp. Math.,
  Vol. 232, pp. 299-311, 1999.

\bibitem{T} C. Toews, Topological components of the set of composition
operators on $H^{\infty}(B_N)$, Integral Equations Operator Theory,
vol. 48, pp. 265-280, 2004.

  \bibitem{TZ2} C. Z. Tong and Z. H. Zhou, Compact intertwining relations for composition operators between
  the weighted Bergman spaces and the weighted Bloch spaces, J. Korean Math. Soc., Vol. 51(1), pp. 125-135, 2014.

  \bibitem{TZ} C. Z. Tong and Z. H. Zhou, Intertwining relations for Volterra operators on the Bergman space,
  Illinois J. Math., to appear.

  \bibitem{X2} J. Xiao, Composition
  operators associated with Bloch-type spaces, Complex Variables, Vol.
  46, pp.109-121, 2001.

  \bibitem{X4} J. Xiao, Riemann-Stieltjes operators between weighted Bergman spaces,
  Proc. Conference on Complex and Harmonic Analysis, Thessaloniki 2007.

  \bibitem{X3} J. Xiao, Riemann-Stieltjes operators on weighted Bloch and Bergman spaces of the unit ball,
  J. London Math. Soc. Vol. 70(2), 199-214, 2004.

  \bibitem{X1} J. Xiao, The $\mathcal{Q}_p$ carelson measure problem, Advances in Mathematics,
  Vol. 217, pp. 2075-2088, 2008.


\bibitem{ZC} Z. H. Zhou and R. Y. Chen, Weighted composition operators fom $F(p, q, s) $to Bloch type spaces,
Internat. J. Math., vol. 19, no. 8, pp. 899-926,
2008.
\bibitem{ZS} Z. H. Zhou and J. H. Shi, Compactness of composition operators on the Bloch space in classical
bounded symmetric domains, Michigan Math. J., vol. 50, pp. 381-405,
2002.
\bibitem{ZZ} H. G. Zeng and Z. H. Zhou, Essential norm estimate of a composition operator between Bloch-type
spaces in the unit ball, Rocky Mountain J. Math. 42(3) (2012), 1049-1071.

  \bibitem{Z2} K. H. Zhu, Spaces of Holomorphic Functions in the Unit
  Ball. Grad. Texts in Math, Springer, 2005.

\end{thebibliography}
\end{document}